\documentclass[12pt]{amsart}
\usepackage{amssymb}
\usepackage{amsmath}
\usepackage{mathtools}  
\usepackage{tabulary}
\usepackage{booktabs}
\usepackage{setspace}
\usepackage{environ}
\usepackage{mathrsfs}
\usepackage{amsfonts}
\usepackage{mathabx}
\usepackage{pict2e}
\usepackage{enumitem}
\usepackage{tikz}
\usepackage{amsthm}
\usepackage{float}
\usepackage{mathpazo}
\usepackage{graphbox}

\usepackage{geometry}
\usepackage{graphicx}
\newtheorem{theorem}{Theorem}[section]
\newtheorem*{tobetheorem*}{Trying to prove}
\newtheorem*{idea*}{Idea}
\newtheorem{lemma}[theorem]{Lemma}

\newtheorem{corollary}[theorem]{Corollary}
\theoremstyle{remark}
\newtheorem{remark}[theorem]{Remark}
\theoremstyle{definition}
\newtheorem{definition}[theorem]{Definition}

\NewDocumentCommand{\mgn}{}{\mathcal M_{g,n}}

\NewDocumentCommand{\sgn}{}{\mathcal S_{g,n}}
\NewDocumentCommand{\mgnb}{}{\overline{\mathcal M}_{g,n}}
\NewDocumentCommand{\bg}{}{\textbf{G}}

\DeclarePairedDelimiterX{\inp}[2]{\langle}{\rangle}{#1, #2}

\DeclareMathOperator{\crit}{Crit}

\DeclareMathOperator{\ind}{ind}
\DeclareMathOperator{\sys}{sys}

\DeclareMathOperator{\syst}{sys_T}

\numberwithin{equation}{section}

\title[Stability phenomena in $\mgnb$ via Morse theory]{Stability phenomena in Deligne--Mumford compactifications via Morse theory}

\author{Changjie Chen}

\begin{document}

\begin{abstract}
    We study the rational homology of the Deligne--Mumford compactification $\overline{\mathcal M}_{g,n}$ of the moduli space of stable curves via a family of Morse functions, namely the $\text{sys}_T$ functions. Exploiting the geometric and Morse properties of $\text{sys}_T$, including the existence of an index gap and additivity of the Morse index upon gluing maps, we reprove that in low degrees the homology of $\overline{\mathcal M}_{g,n}$ is supported entirely on the boundary $\partial \overline{\mathcal M}_{g,n}$, providing a geometric perspective complementary to Harer’s classical result on the virtual cohomological dimension.

    Furthermore, we establish finite generation and stability phenomena for the rational homology across all genera and numbers of marked points. We show that for each degree $k$, a finite set of homology elements generates all $k$-th homology classes via attaching copies of thrice-marked $\mathbb{P}^1$. This result also recovers previously known stability in the number of marked points, such as Tosteson's theorem.

\end{abstract}

\maketitle

\section{Introduction}
% We present a stability phenomenon in the homology of Deligne--Mumford compactifications $\mgnb$, in the form that all homology of them of a given degree $k$ is finitely generated

The moduli space of stable curves $\mgnb$ is a central object in algebraic geometry and geometric topology. Its rich structure combines complex analytic, hyperbolic, and combinatorial perspectives. Understanding the topology of $\mgnb$, particularly its rational homology in various degrees, remains a fundamental question.

In this paper, we study the low-degree homology and stability phenomena of $\mgnb$ via the \emph{$\syst$ functions}, a family of $C^2$-Morse functions. These functions, inspired by systolic geometry, encode geometric information about closed geodesics on hyperbolic surfaces and continuously extend to the Deligne--Mumford compactification. The $\syst$-Morse theoretic techniques combine geometric, combinatorial, and Morse-theoretic ideas, offering a novel approach to understanding the topology of moduli spaces of curves.

Our first result concerns the low-degree homology of $\mgnb$. By exploiting the $\syst$-Morse theory and the structure of the Deligne--Mumford boundary, we reprove:

\begin{theorem}[Low-degree homology is supported on the boundary, Theorem~\ref{isomorphism of low degree homology}]
    For any $k>0$, when $g+n$ is sufficiently large, in particular when $\ind(g,n)>k+1$, the inclusion
    \[
    i\colon \partial\mgnb\to\mgnb
    \]
    induces an isomorphism
    \[
    i_*\colon H_k(\partial \mgnb;\mathbb Q) \xrightarrow{\cong} H_k(\mgnb;\mathbb Q).
    \]
\end{theorem}

This theorem implies that, in the stable range, low-degree homology classes are determined by boundary strata, a phenomenon revealed by Harer's theorem on the \emph{virtual cohomological dimension} of $\mgn$ in \cite{harer1986virtual}.

Our second main result addresses \emph{stability phenomena} across the moduli spaces $\mgnb$ as $(g,n)$ varies. Let $\ind(g,n)$ be the lowest Morse index of $\syst$-critical points in $\mgn$. Using the combinatorial framework of stable graphs and the $\syst$-Morse theory, we show:

\begin{theorem}[Finite generation of homology over attaching, Theorem~\ref{finite generation}]
    For every $k\ge0$, there exists a finite set
    \[
    \Xi=\{\xi_i\in H_k(\overline{\mathcal M}_{g_i,n_i};\mathbb Q)\}
    \]
    that generates
    \[
    \bigoplus_{(g,n)} H_k(\mgnb;\mathbb Q)
    \]
    via attaching thrice-marked $\mathbb P^1$'s. Here, the stable range is given by $\ind(g,n)>k$.
\end{theorem}

This result provides a concrete description of the generation of homology by boundary strata. It recovers previously known stability results in $n$, the number of marked points, such as Tosteson's theorem \cite{tosteson2021stability}. 

The structure of the paper is as follows. In Section~\ref{sec:stable-graphs}, we review stable graphs, strata of $\mgnb$, and the definition of the $\syst$ functions. Section~\ref{sec:wp-metric} discusses the Weil--Petersson metric and potential substitutes for the $\syst$ functions. Section~\ref{sec:low-degree-homology} proves the low-degree homology isomorphism, and Section~\ref{sec:stability} establishes the finite generation and stability phenomena for the rational homology of $\mgnb$. Finally, in Section~\ref{others work}, we relate our results to previous work.

\subsection{Acknowledgements:}
The author is grateful to Louis Hainaut, Siddarth Kannan, Jeremy Kahn, and Philip Tosteson for valuable conversations and insights that contributed to the development of this work. The author thanks Oscar Randal-Williams for thoughtful comments and suggestions.

\section{Stable graphs and the $\syst$ functions}
\label{sec:stable-graphs}

We review standard definitions used to study Deligne--Mumford compactifications and discuss the construction and geometric and Morse properties of the $\syst$ functions.

\subsection{Stable graph and Deligne--Mumford compactification}

We follow standard definitions in the literature on Deligne--Mumford compactifications (see, e.g., \cite{deligne1969irreducibility,chan2021moduli}).

\begin{definition}
Let $X$ be a stable curve with marked points $p_1,\dots, p_n$ over $\mathbb{C}$. The \emph{stable graph $\bg=(G,g)$ dual to $X$} is a weighted multi-graph consisting of the following data:

\begin{itemize}
    \item The set $V(G)$ of vertices: each $v$ corresponds to an irreducible component $X_v$.
    \item The set $E(G)$ of edges: each $e=e(v,w)$ corresponds to a node connecting $X_v$ and $X_w$.
    \item The set $H(G)$ of half-edges: each $h=h(v)$ corresponds to a marked point on $X_v$.
    \item A genus function $g\colon V(G) \to \mathbb N$ sending $v$ to the genus of the normalization of $X_v$.
    \item A marking function $m\colon \{1,\dots,n\}\to V(G)$ sending $i$ to $v$ when $p_i$ lies in $X_v$.
\end{itemize}
\end{definition}

One may define an abstract stable graph in the same way as long as it satisfies the stability condition
\[
2g(v)-2+n(v)>0,
\]
where $n(v)=\# m^{-1}(v)$.

Conversely, given an abstract stable graph satisfying the stability condition, one can construct a dual stable curve. The \emph{stratum $\mathcal M_\bg$ dual to a stable graph $\bg$} is the collection of all stable curves dual to the same stable graph. Formally, the stratum $\mathcal M_\bg$ is given by
\[
\mathcal M_\bg=\prod_{v\in V(G)} \mathcal M_{g(v),n(v)}/\text{Aut}(\bg),
\]
where $\text{Aut}(\bg)$ denotes the automorphism group of the stable graph $\bg$.

\subsection{The Morse functions}

There is a canonical correspondence between algebraic structures and hyperbolic structures on a topological real surface. Thus, the moduli space $\mgn$ can be defined as the quotient of the space of hyperbolic metrics by the action of orientation-preserving diffeomorphisms.

\begin{definition}
\label{syst definition}
Let $0<T<1$ be a parameter. For a hyperbolic structure $X\in\mgn$, define
\[
\syst(X)=-T\log\left(\sum_{\gamma \text{ s.c.g. on } X} e^{-\frac1Tl_\gamma(X)}\right),
\]
where s.c.g. stands for simple closed geodesic. The definition of it on the Deligne--Mumford compactification $\mgnb$ is by continuous extension.
\end{definition}

\begin{definition}
\label{eutacticity definition}
    A point $X$ in the Teichmüller space $\mathcal T$ is \emph{eutactic} if the origin lies in the interior of the convex hull of $\{\nabla l_\gamma\}_{\gamma\in S(X)}$ in the tangent space $T_X\mathcal T$, with respect to the subspace topology, where $S(X)$ is the set of shortest closed geodesics on $X$, $l_\gamma$ is the geodesic-length function associated to $\gamma$, and $\nabla$ is the (Weil--Petersson) gradient vector. In this case, the \emph{eutactic rank} of $X$ is defined as the rank of the set ${\nabla l_\gamma}_{\gamma\in S(X)}$.
\end{definition}

\begin{definition}
    For a set $B=\{\beta_1,\dots,\beta_k\}$ of pairwise disjoint essential closed geodesics on a given hyperbolic surface, let $D_B\subset\mgnb$ be the stratum where the curves in $B$ are pinched. Extend $B$ to a pants decomposition and consider the standard Fenchel--Nielsen coordinates $(l_\gamma,\tau_\gamma)_\gamma$. For $X\in D_B$, define the local chart
    \[
    (l_\gamma^{\chi(\gamma)},\tau_\gamma)_\gamma/\text{Aut}(X),
    \]
    where $\chi(\gamma)=\frac{1}{2}$ if $\gamma\in B$ and 1 otherwise. This endows $\mgnb$ with the the \emph{root-geodesic-length smooth orbifold structure} on $\mgnb$.
\end{definition}

A pointed subspace of $\mgnb$ is called \emph{stratifically stable} if the subspace is contained in the relative Deligne--Mumford compactification of the stratum the point lies in. The following theorem establishes the Morse property of $\syst$, which is foundational to subsequent Morse homology results.

\begin{theorem}[\cite{chen2023morse}]
\label{Morse property}
For sufficiently small $T>0$, we have:
\begin{itemize}
    \item The $\syst$ function is $C^2$-Morse on the Deligne--Mumford compactification $\overline{\mathcal M}_{g,n}$, equipped with the root-geodesic-length smooth orbifold structure.

    \item In the moduli space $\mathcal M_{g,n}$, $\syst$-critical points $p_T$ correspond to eutactic points $p$, with Weil--Petersson distance $d_{WP}(p_T,p) < C T$, and Morse index $\ind(p_T)$ equals the eutactic rank of $p$.

    \item For $X\in\mathcal M_{\textbf{G}}\subset\mgnb$, the following are equivalent:

    (1) $X$ is a critical point of $\syst$ on $\mgnb$;
    
    (2) For any $v\in V(G)$, $X_v$ is a critical point of $\syst$ on $\mathcal M_{g(v),n(v)}$.
    
    And in this case,
    \[
    \ind(X)=\sum_{v\in V(\textbf{G})}\ind(X_v).
    \]

    \item The Weil--Petersson gradient flow of $\syst$ on $\overline{\mathcal M}_{g,n}$ is well defined, and the unstable manifold at any point is stratifically stable. The closure of any $\syst$-flow line is contained in a single stratum or hits a critical point downward in its boundary stratum transversely.
\end{itemize}
\end{theorem}

The transversality statement in the last item strengthens the original theorem and is established as Lemma~\ref{transverse to submanifold}.

An important property of the $\syst$ functions is the existence of an index gap.

\begin{definition}
    Let $\ind(g,n)$ be the lowest Morse index of $\syst$-critical points in $\mgn$, that is,
    \[
    \ind(g,n)=\min_{X\in\crit(\syst|_{\mgn})}\ind(X).
    \]
\end{definition}

\begin{theorem}[\cite{chen2023index}]
\label{index gap theorem}
There exists a universal constant $C>0$ such that the lowest Morse index satisfies
\[
\ind(g,n) > C \log\log(g+n).
\]
\end{theorem}

\begin{remark}
\label{limitation}
    At present, this Morse-theoretic approach does not yield optimal numerical results for the (co)homology, primarily because relatively little is known about the following aspects:

    (1) The lower bound on $\ind(g,n)$ given above is not sharp. As shown in \cite{bourque2020hyperbolic}, a linear bound cannot be achieved, in contrast with many results such as the stable range for the homology of $\mgn$. Nevertheless, one might still expect a logarithmic bound.

    (2) In general, the number of critical points exceeds the dimension of the homology in the corresponding degree. When a Riemannian metric is chosen, the critical points interact with each other via gradient flow lines, and the correspondence between critical points and homology generators remains largely unexplored.
\end{remark}

\section{Subtleties in using the Weil--Petersson metric}
\label{sec:wp-metric}

\subsection{Weil--Petersson metric near the boundary}

The Weil--Petersson metric, induced by the natural $L^2$-pairing on the space of (0,2)- or (-1,1)-tensors, is known to be incomplete on Teichmüller space. Consequently, when extended to the Deligne--Mumford compactification $\mgnb$, it defines an orbifold Riemannian metric that is singular on the boundary $\partial\mgnb$. Wolpert \cite{wolpert2009extension} computed the expansion of the Weil--Petersson connection near the boundary of $\mgnb$.

The degeneracy of the Weil--Petersson metric may pose challenges when applying Morse theory, even after small perturbations. Nevertheless, it is expected to be feasible with careful attention to finiteness conditions. For convenience, we consider a globally well-defined ``twin'' metric that shares key features with the Weil--Petersson metric in the Morse sense.

Let $S=\{\gamma_i\}$ be a set of mutually disjoint simple closed geodesics, and let $\mathcal S_S$ denote the stratum obtained by pinching all elements of $S$. Masur \cite{masur1976extension} showed that, near $\mathcal S_S$, the Weil--Petersson metric admits the expansion
\[
g_{\text{WP}}=\frac{\pi}{2}\sum_{i}\left( \frac{dl_i^2}{l_i}+\frac{l_i^3}{4\pi^4}d\tau_i^2 \right)+g_{\text{WP}}^{\mathcal S}+\text{higher order terms},
\]
where $(l_i,\tau_i)_{i}$ are the Fenchel--Nielsen coordinates associated to $S$, and $g_{\text{WP}}^{\mathcal S}$ is the Weil--Petersson metric on the stratum $\mathcal S_S$. Hence, up to first order, the tangent subspace normal to the stratum $\mathcal S_S$ is spanned by the coordinates $\{(l_i,\tau_i)\}$.

\begin{definition}[\cite{chen2023morse}]
    The \emph{root-geodesic-length} differential orbifold structure of $\mgnb$ is defined by assigning local charts
    \[
    (l_i^\frac{1}{2},\tau_i)_{i=1}^k + \text{Fenchel--Nielsen coordinates for $S^C$ in the surface}
    \]
    near the stratum $\mathcal S_S$ where $S=\{\gamma_1,\dots,\gamma_k\}$.
\end{definition}

Under the root-geodesic-length differential structure, the Weil--Petersson metric is well defined in the radial directions but blows up in the angular directions with respect to $S$ near $\mathcal S_S$.

\subsection{Gradient systems}

We use the results of Bárta, Chill, and Fašangová \cite{barta2012every} and Bílý \cite{bily2014transformations}, which guarantee the existence of a Riemannian metric that turns a \emph{gradient-like} vector field into an actual gradient vector field, both with respect to a given differentiable function.

Let $V$ be a continuous vector field on a manifold $M$. A \emph{stationary point} of $V$ is where $V$ vanishes.

\begin{definition}
    A $C^1$-continuous function $f\colon M\to\mathbb R$ is called a \emph{strict Lyapunov function} for the ordinary differential equation
    \[
    \dot{x} + V(x) = 0,
    \]
    if
    \[
    \inp{f'(x)}{V(x)}>0,
    \]
    whenever $V(x)\neq0$.
\end{definition}

Bárta, Chill, and Fašangová \cite{barta2012every} proved the strict Lyapunov condition is sufficient for the existence of such a Riemannian metric away from the stationary points of $V$, and Bílý \cite{bily2014transformations} gave a condition when it admits a global extension.

\begin{theorem}[\cite{barta2012every,bily2014transformations}]
    Let $V$ be a continuous vector field on a manifold $M$, and let $f\colon M\to \mathbb R$ be a $C^1$-continuous strict Lyapunov function for
    \[
    \dot{x} + V(x) = 0,
    \]
    then there exists a Riemannian metric on the complement of the stationary points of $V$.

    Furthermore, if at every stationary point $x$ of $V$, $V'$ is regular, and $f''V'^{-1}$ defines a scalar product, then there exists such a Riemannian metric on $M$.
\end{theorem}

In our setting, $V=\nabla^\text{WP}\syst$, which is the stratumwise Weil--Petersson gradient vector of $\syst$, and $f=\syst$. It is easy to check that the conditions are satisfied and therefore, there exists an orbifold Riemannian metric on $\mgnb$ such that the gradient vector of the $\syst$ function coincides with $\nabla^\text{WP}\syst$.

It is clear that this modification of the metric does not affect the Morse properties of the $\syst$ functions, including their critical points and Morse indices.

\begin{remark}

    The choice of such a metric is not unique. Given a (Morse) function $f$ and a Riemannian metric $g$, let $K=\ker \nabla_g f$. Modifying $g$ on $K\times K \subset TM \times TM$ leaves the gradient vector of $f$ unchanged. In the case of the Weil--Petersson metric, which diverges only in twist directions near boundary strata, the stratumwise gradient of $\syst$ defines a global $C^1$ vector field transverse to these directions. By suitably modifying the metric in the twist directions, one obtains a globally well-defined Riemannian metric without changing the gradient vector field.
\end{remark}

\begin{lemma}
    Let $f\colon M\to\mathbb R$ be a Morse function on a manifold $M$ equipped with a Riemannian metric, and $N\subset M$ a Morse submanifold of $M$. Then if the closure of a flow line of $f$ intersects $N$, then it is contained in $N$ or hits a critical point in $N$ transversely.
\end{lemma}

\begin{proof}
    Let $\varphi$ be the closure of an $f$-flow line. If $\varphi$ contains a point in $N$ in its interior, then $\varphi$ is contained in $N$ since $N$ is a Morse submanifold of $M$. Suppose $\varphi$ hits a point $p\in N$ from outside $N$, then $p$ lies at the end of $\varphi$ and is a critical point of $f$. Take the orthogonal decomposition of the tangent space as
    \[
    T_pM=T_pN\oplus \nu_pN.
    \]
    Let $(x,y)$ be a compatible local chart around $p$ with $(x,y)(p)=(0,0)$ such that
    \[
    f(x,y)=f(p)+x^TAx+y^TBy+x^TCy+g(x,y)
    \]
    where $A$ and $B$ are nondegenerate matrices, and $\nabla^2 g=0$. Note that
    \[
    \nabla f=(2Ax+Cy,2By+Cx),
    \]
    and therefore, $C=0$ since locally $N=\{y=0\}$.

    Pick $p_0=\varphi(t_0)=(x_0,y_0)$ where $y_0\neq0$, then the flow through $p_0$ is given by
    \[
    \varphi(t-t_0)=(e^{2At}x_0,e^{2Bt}y_0),
    \]
    which is transverse to $N$ if it converges to $p$ as $t\to-\infty$.
\end{proof}

\begin{corollary}
\label{transverse to submanifold}
    The closure of any $\syst$-flow line, with respect to the stratumwise Weil--Petersson metric, is contained in a single stratum or hits a critical point downward in its boundary stratum transversely.
\end{corollary}

\section{Low degree homology}
\label{sec:low-degree-homology}

The Deligne--Mumford compactification $\mgnb$ is a compact orbifold. The homology of such a space can be related to the homology of any finite manifold cover via the action of the deck transformation group, cf. \cite{hatcher_algebraic_2001}.

\begin{lemma}
Let $\pi\colon \widetilde{M}\to M$ be a finite manifold cover with deck group $G$ acting on $\widetilde{M}$. Then
\[
H_*(M;\mathbb Q) \cong H_*(\widetilde{M};\mathbb Q)^G,
\]
where the superscript denotes the $G$-invariant part of homology.
\end{lemma}

\begin{theorem}
\label{isomorphism of low degree homology}
    For any $k>0$, when $g+n$ is sufficiently large, in particular when $\ind(g,n)>k+1$, the inclusion
    \[
    i\colon \partial\mgnb\to\mgnb
    \]
    induces an isomorphism
    \[
    i_*\colon H_k(\partial \mgnb;\mathbb Q) \xrightarrow{\cong} H_k(\mgnb;\mathbb Q).
    \]
\end{theorem}

The proof uses the following two estimates from \cite{chen2023morse}:
\begin{itemize}
    \item It is uniform on $\mgnb$ that
    \[
    |\syst-\sys|<CT.
    \]
    \item Near a boundary stratum $D_S$, where $S=\{\gamma_i\}$ is a set of mutually disjoint simple closed geodesics, we have
    \[
    \inp{\nabla^{\text{WP}}\syst}{\nabla l_i}>0.
    \]
\end{itemize}

\begin{proof}
Let $M$ be a finite manifold cover of $\mgnb$, and let $B \subset M$ denote the preimage of $\partial\mgnb$. By abuse of notation, we denote the lift of the $\syst$ function to $M$ also by $\syst$.  

Since $|\syst - \sys| < C T$, for $\delta>0$ sufficiently small we have
\[
\syst^{-1}([0,\delta]) \subset M^{\le\epsilon},
\]
where $M^{\le\epsilon}$ is the preimage of the $\epsilon$-thin part $\mgnb^{\le\epsilon} = \sys^{-1}([0,\epsilon])$. In particular,
\[
\crit(\syst) \cap \syst^{-1}([0,\delta)) \subset B.
\]

By Morse handle decomposition \cite{morse1959topologically,milnor2016morse}, $M$ is obtained from $\syst^{-1}([0,\delta])$ by attaching handles of index $m = \ind(g,n), \dots, \dim M$.  

Since the level sets $\sys^{-1}(c)$ are transverse to $\nabla^\text{WP}\syst$ for small $c>0$, we have homeomorphisms
\[
B\times [0,\delta] = \sys^{-1}(0)\times [0,\delta] \cong M^{\le\epsilon} \cong \syst^{-1}([0,\delta]).
\]

It follows that there is an isomorphism on homology in low degrees:
\[
i_* \colon H_k(B) \xrightarrow{\cong} H_k(M),
\]
and hence, passing to the quotient by the deck group,
\[
i_* \colon H_k(\partial\mgnb;\mathbb Q) \xrightarrow{\cong} H_k(\mgnb;\mathbb Q).
\]
\end{proof}

Let $H^\text{BM}_*$ denote the Borel--Moore homology. Taking the long exact sequence of the pair $(\mgnb, \partial\mgnb)$, we deduce:

\begin{theorem}
Under the same assumptions, we have
\[
H^{6g-6+2n-k}(\mgn;\mathbb Q) \cong H^\text{BM}_k(\mgn;\mathbb Q) \cong H_k(\mgnb, \partial\mgnb;\mathbb Q) = 0.
\]
% \[
% H^\text{BM}_k(\mgn;\mathbb Q)\cong H^{6g-6+2n-k}(\mgn;\mathbb Q) \cong H^c_k(\mgn;\mathbb Q) \cong H_k(\mgnb, \partial\mgn;\mathbb Q) = 0.
% \]
\end{theorem}

\begin{remark}
This result is equivalent to Harer's theorem on the virtual cohomological dimension \cite{harer1986virtual} in the stable range, although here we only consider a weaker statement for sufficiently large $g+n$.
\end{remark}

\section{Stability Phenomena}
\label{sec:stability}

To analyze homology and stability phenomena more precisely, we introduce the following definition.

\begin{definition}
    Let $\sgn(k)$ denote the collection of Deligne--Mumford strata of $\mgnb$ containing $\syst$-critical points of index $k$. Define
    \[
    \sgn([k]) \subset \mgnb
    \]
    as the relative Deligne--Mumford compactification of the union $\bigcup_{j\le k} \sgn(j)$.
\end{definition}

By the index gap theorem (Theorem~\ref{index gap theorem}), for $k<\ind(g,n)$ we have
\[
\sgn([k]) \subset \partial\mgnb.
\]
Analogously to Theorem~\ref{isomorphism of low degree homology}, the low-degree homology of $\mgnb$ arises from the strata in $\sgn([k])$.

\begin{definition}
For a stable graph $\bg$, define its \emph{stratum dimension} $\dim \bg$ as the dimension of $\mathcal M_\bg$, and its \emph{local stratum dimension} $|\bg|$ as
\[
|\bg| = \max_{v \in V(G)} \dim_{\mathbb R} \mathcal M_{g(v),n(v)}.
\]
\end{definition}

\begin{lemma} 
For $k<\ind(g,n)$, there exist constants $v=v(k)>0$ and $h=h(k)>0$ such that for any $\mathcal M_\bg \in \sgn([k])$, one has
\begin{align*}
    &|\bg| \le h,\\
    &\#V(G) \ge v,\\
    &\#\{v\in V(G) : \dim \mathcal M_{g(v),n(v)}>0 \} \le \frac{1}{2} hk.
\end{align*}
\end{lemma}

\begin{proof}
If the first two inequalities hold for a stratum $\mathcal S$, they also hold for any boundary stratum.  

For any $X \in \mathcal M_\bg \in \sgn(k)$, we have
\[
\ind(g(v),n(v)) \le \ind(X_v) \le \ind(X) = k,
\]
so $\dim_{\mathbb R} \mathcal M_{g(v),n(v)}$ is bounded for each $v\in V(G)$. The lower bound on $\#V(G)$ follows from
\[
\#V(G) \ge \frac{6g-6+2n}{|\bg|}.
\]

By the classification of low-index critical points \cite{chen2023index}, we have
\[
\ind(g(v),n(v)) \ge 1 \quad \text{whenever } \dim \mathcal M_{g(v),n(v)}>0.
\]
Using the additivity of Morse indices over $V(G)$, it follows that
\[
\#\{v\in V(G) : \dim \mathcal M_{g(v),n(v)}>0 \} \le k.
\]
The same inequalities extend to boundary strata of $\mathcal M_\bg$.
\end{proof}

In other words, when $\syst$-Morse theory applies, the `complexity' of $k$-th rational homology elements stabilizes as either $g$ or $n$ increases.

\begin{corollary}
\label{finite generation morse}
For every $k \ge 0$, there exists a finite set of $\syst$-critical points
\[
S = \{X_i \in \mathcal M_{g_i,n_i}\}
\]
that generates all critical points of index $k$ in $\bigsqcup_{(g,n)} \mgnb$ via the attachment of copies of thrice-marked $\mathbb P^1$.
\end{corollary}

Let $\textbf{P}$ denote the stable graph dual to $\mathcal M_{0,3}$. For a permutation $\sigma \in S_n$, define $\sigma$-attaching maps $\text{Att}_1(\sigma), \text{Att}_2(\sigma), \text{Att}_3(\sigma)$ on the set of stable graphs $\bg$ as follows. Each cycle of $\sigma$ is assigned to a single vertex of $\bg$ via the marking function, and for $i=1,2,3$, $\text{Att}_i(\sigma)(\bg)$ is obtained by connecting the half-edges of $\textbf{P}$ to $\bg$ so that
\[
m_{\text{Att}_i(\sigma)(\bg)}(n') = m_\bg(\sigma(n')), \quad n' = 1,\dots,n-i,
\]
and
\[
m_{\text{Att}_i(\sigma)(\bg)}(n-i+j) = m_\textbf{P}(1+j), \quad j=1,\dots,i.
\]

In this framework, Part III of Theorem~\ref{Morse property} implies:

\begin{lemma}
\label{in the image}
    For any $k$, except for finitely many pairs $(g,n)$, for any $\xi\in H_k(\mgnb;\mathbb Q)$, there exists $\sigma\in S_n$, such that
    \[
    \xi\in\text{im}((\text{Att}_i(\sigma)_*:H_k(\mgnb;\mathbb Q)\to H_k(\overline{\mathcal M}_{g+i-1,n+3-2i};\mathbb Q)).
    \]
\end{lemma}

By composing attaching maps, one can attach a stable graph of stratum dimension $0$ (dual to a maximally noded curve) to stable graphs with marked points. Morse-theoretically, attaching copies of thrice-marked $\mathbb P^1$ preserves criticality and the Morse index, inducing equivariant homomorphisms on rational homology and Morse homology.

Combining Corollary~\ref{finite generation morse} and Lemma~\ref{in the image}, we obtain:

\begin{theorem}
\label{finite generation}
    For every $k\ge0$, there exists a finite set
    \[
    \Xi=\{\xi_i\in H_k(\overline{\mathcal M}_{g_i,n_i};\mathbb Q)\}
    \]
    that generates
    \[
    \bigoplus_{(g,n)} H_k(\mgnb;\mathbb Q)
    \]
    via attaching thrice-marked $\mathbb P^1$'s. Here, the stable range is given by $\ind(g,n)>k$.
\end{theorem}

\section{Relation to previous work}
\label{others work}

\subsection{Tosteson's stability theorem in $n$}

Tosteson \cite{tosteson2021stability} studies the asymptotic behavior of the homology of $\mgnb$ using $\mathbf{FS}^\text{op}$-modules.

Let $\mathbf{FS}$ be the category of finite sets and surjections. An $\mathbf{FS}^\text{op}$-module is a functor from $\mathbf{FS}^\text{op}$ to the category of vector spaces, denoted $n\mapsto V_n$, and it is \emph{finitely generated in degree $\le C$} if there is a finite list of classes $\{v_i\in V_{d_i}\}$ with $d_i\le C$, such that every $V_n$ is spanned by classes of the form $f^*v_i$.

\begin{theorem}[\cite{tosteson2021stability}]
    Let $g,i\in\mathbb N$. Then the $\mathbf{FS}^\text{op}$-module
    \[
    n\mapsto H_i(\mgnb;\mathbb Q)
    \]
    is a subquotient of an extension of $\mathbf{FS}^\text{op}$-modules that are finitely generated in degree $\le 8g^2i^2+29g^2i+16gi^2+21g^2+10gi-6g$.
\end{theorem}

Furthermore, Tosteson considers the functor $\text{Res}_r$ that restricts an $\mathbf{FS}^\text{op}$-module to an $\mathbf{FS}_r^\text{op}$-module, for $r\in\mathbb N$, where $\mathbf{FS}_r^\text{op}$ is the full subcategory of $\mathbf{FS}^\text{op}$ spanned by sets of size $\le r$. The functor $\text{Res}_r$ has a left adjoint $\text{Ind}_r$, which takes an $\mathbf{FS}_r^\text{op}$ module to the $\mathbf{FS}^\text{op}$ module freely generated by it modulo relations in degree $r$.

\begin{theorem}[\cite{tosteson2021stability}]
    Let $i, g\in\mathbb N$. There exists $N\in\mathbb N$ such that the natural map of $\mathbf{FS}^\text{op}$-modules
    \[
    \text{Ind}_N \text{Res}_N H_i(\overline{\mathcal M}_{g,-}, \mathbb Q) \to H_i(\overline{\mathcal M}_{g,-}, \mathbb Q)
    \]
    is an isomorphism. In particular, any presentation of the $\mathbf{FS}_N^\text{op}$-module $\text{Res}_N H_i(\overline{\mathcal M}_{g,-}, \mathbb Q)$ gives a presentation of the $\mathbf{FS}^\text{op}$-module $H_i(\overline{\mathcal M}_{g,-}, \mathbb Q)$.
\end{theorem}

\begin{remark}
    Our Theorem~\ref{finite generation} recovers Tosteson's stability theorems in $n$, noting that a Whitehead move (or WDVV exchange), on a stable graph preserves homology up to a nonzero factor. However, our bound on the stable range is weaker than Tosteson's, in the case of fixed $g$, c.f. Remark \ref{limitation}.
\end{remark}

\subsection{Tautological extension}

A \emph{tautological morphism} is a morphism induced by gluing, forgetting, and permuting marked points. The \emph{tautological ring} $RH^*(\mgnb)\subset H^*(\mgnb)$ is the smallest subring that is closed under all tautological morphisms.

Canning, Larson, and Payne \cite{canning2024extensions} define a \emph{semi-tautological extension (STE)} as a collection of subrings $S^*(\mgnb)$ of $H^*(\mgnb)$ that contains the tautological subrings $RH^*(\mgnb)$ and is closed under pullback by forgetting and permuting marked points and under pushforward for gluing marked points. They prove
\begin{theorem}[\cite{canning2024extensions}]
    For any fixed degree $k$, the STE generated by
    \[
    \{H^{k'}(\overline{\mathcal M}_{g',n'}):k'\le k, g'<\frac{3}{2}k'+1,n'\le k', 4g'-4+n'\ge k\}
    \]
    contains $H^k(\mgnb)$ for all $g$ and $n$.
\end{theorem}

In our setting, we only consider tautological morphisms of certain types, namely gluing (G) and permuting (P) marked points; forgetful maps are not included. Analogous to the perspective of Canning, Larson, and Payne, we may consider \emph{tautological extensions of type `GP'} in homology. Our Theorem \ref{finite generation} then says that a finitely generated tautological extension of type `GP' contains $H_k(\mgnb;\mathbb Q)$ for all $(g,n)$.

% The \emph{tautological subgroup} of $RH_*(\mgnb)\subset H_*^\text{BM}(\mgnb)$ is the smallest subspace that is closed under tautological morphisms. We consider

% Let the \emph{narrow tautological closure} of a set $\Xi$ be the smallest collection of subspaces that is closed under pushforward for the tautological morphisms induced by gluing and permuting marked points.

% \begin{theorem}
%     For any degree $k$, there exits a finite set
%     \[
%     \Xi=\{\xi_i\in H_k(\overline{\mathcal M}_{g_i,n_i};\mathbb Q)\}
%     \]
%     whose narrow tautological closure contains $H_k(\mgnb;\mathbb Q)$ for all $(g,n)$.
% \end{theorem}

\clearpage
\bibliographystyle{alpha}
\bibliography{main}

\end{document}